\documentclass[11pt]{article}

\usepackage[utf8]{inputenc} 
\usepackage[colorlinks,citecolor=blue]{hyperref}
\usepackage{latexsym,amsmath,amssymb,textcomp}
\usepackage[all]{xy}
\usepackage[nottoc, notlof, notlot]{tocbibind}
\usepackage{geometry} 
\usepackage{graphicx} 
\usepackage{amsthm}
\usepackage{amsmath}
\usepackage{booktabs}
\usepackage{array} 
\usepackage{hyperref}

\title{On derived equivalences for categories of generalized intervals of a finite poset}
\author{Fr\'ed\'eric Chapoton,  Sefi Ladkani and Baptiste Rognerud}
\begin{document}
\maketitle

%
%
\theoremstyle{plain}
\newtheorem{theo}{Theorem}[section]
\newtheorem{prop}[theo]{Proposition}
\newtheorem{lemma}[theo]{Lemma}
\newtheorem{coro}[theo]{Corollary}
\newtheorem{conj}[theo]{Conjecture}
\newtheorem{question}[theo]{Question}
\newtheorem*{notations}{Notations}
\theoremstyle{definition}
\newtheorem{de}[theo]{Definition}
\theoremstyle{remark}
\newtheorem{ex}[theo]{Example}
\newtheorem{re}[theo]{Remark}
\newtheorem{res}[theo]{Remarks}

%
%

\renewcommand{\labelitemi}{$\bullet$}
\newcommand{\comu}{co\mu}
\newcommand{\mo}{\ensuremath{\hbox{mod}}}
\newcommand{\Mo}{\ensuremath{\hbox{Mod}}}
\newcommand{\decale}[1]{\raisebox{-2ex}{$#1$}}
\newcommand{\decaleb}[2]{\raisebox{#1}{$#2$}}
\newcommand\field {{k}}
\newcommand\Mod {\hbox{-}{\rm Mod}}
\newcommand\Hom {{\rm Hom}}
\newcommand\End {{\rm End}}
\newcommand\Out{{\rm Out}}
\newcommand\mb{\hbox{-}}
\newcommand\D{\mathcal{D}}
\newcommand\C{\mathcal{C}}
\newcommand\Rad{{\rm Rad}}
\newcommand\Top{{\rm Top}}
\newcommand\Int{\mathrm{Int}}
\newcommand\U{\mathrm{U}}
\newcommand\FYk{\mathcal{F}_{Y,k}}
\newcommand\FGk{\mathcal{F}_{\Gamma,k}}
\newcommand\Fyk{\mathcal{F}_{F(y),k}}
\newcommand\FUk{\mathcal{F}_{\mathrm{U}_x,k}}
\newcommand\Fintk{\mathcal{F}_{\mathrm{Int(X)},k}}
\newcommand\ev{\mathrm{ev}}
\newcommand\Si{\mathrm{S}}
  \def\commutatif{\ar@{}[rd]|{\circlearrowleft}}
\begin{abstract}
We study two constructions related to the intervals of finite posets. The first one is a poset. The second one is more complicated. Loosely speaking it can be seen as a poset with some extra zero-relations.  As main result, we show that these two constructions are equivalent at the level of derived categories. 
\end{abstract}

\par\noindent
{\it {\footnotesize A.M.S. subject classification: 05E10, 16G20}}
\section{Introduction}

Let $k$ be a commutative ring with unit. There are various equivalent definitions of the notion of representation of a poset over $k$. One can look at representations of the \emph{Hasse diagram} of the poset viewed as a quiver with the relations of total commutativity, or at modules over the so-called \emph{incidence algebra} of the poset. An alternative definition is to use a \emph{functor category}. To a poset $(Y,\leqslant)$ one can associate a finite category $\mathcal{C}_Y$ where the objects are the elements of $Y$ and there is a unique morphism between $y$ and $y'$ if and only if $y\leqslant y'$. It is well-known that the category of covariant functors from $\mathcal{C}_Y$ to the category of $k$-vector spaces is equivalent to the category of right modules over the incidence algebra of $Y$. Moreover, it is also classical that if $k\mathcal{C}_Y$ is the $k$-linearization of $\mathcal{C}_Y$, then the category of functors from $\mathcal{C}_Y$ to the category of all $k$-modules $k\Mod$ is equivalent to the category of \emph{$k$-linear functors} from $k\mathcal{C}_Y$ to $k\Mod$.
\begin{de}
Let $k$ be a commutative ring and $(Y,\leqslant )$ be a poset. The category of $k$-linear functors from $k\mathcal{C}_Y$ to $k\Mod$ is denoted $\mathcal{F}_{Y,k}$. \end{de}

Note that the study of these functor categories is very
different from another kind of ``representation of posets", that
was considered by Nazarova, Kleiner and others (See \cite{simson} for more details), which involves a
non-abelian subcategory of the category of modules over a
one-point extension of the Hasse diagram.

Since the category $k\Mod$ is abelian, the category $\FYk$ inherits an abelian structure. In particular, one can consider the bounded derived category $D^b(\FYk)$ of this abelian category. With a slight abuse of notation, we call it the derived category of the poset. Numerous invariants of the poset can be read inside the derived category such as its cardinality or its number of connected components. It is also the good setting for the study of the Coxeter transformation and the Coxeter polynomial of finite posets (See \cite{ladkani_derived_poset} for more details). If two finite posets share the same derived category, then they have the same Coxeter polynonial. Using a computer, it is then easy to find many examples of finite posets with the same Coxeter polynomial, and one can wonder if they also share the same derived category.

In this spirit, there is an interesting conjectural example in the theory of Tamari lattices. It is conjectured by the first author that the Tamari lattice is derived equivalent to the poset of Dyck paths (See \cite{chapoton_derived_tamari} for more details). In the same context, we propose another conjecture involving the derived category of the poset of Dyck paths (See Conjecture \ref{conj_dyck}).

There are various tools that can be used to check if two posets share the same derived category (see \cite{ladkani_derived_poset} or \cite{ladkani_universal} for some explicit constructions). Unfortunately, there are no (known) algorithm and it is most of the time difficult to build such derived equivalences.

One of the difficulties comes from the fact that the derived category of a finite poset may also be equivalent to the derived category of a ring with a-priori no relation with posets. For example the poset $1<2 < 3$ is derived equivalent to the quotient of the path algebra of the quiver $1\to 2 \to 3$ by the ideal generated by the path of length two.

In this article we will focus on the set of intervals of finite posets. We start with two possible definitions of categories of intervals of a poset. The first definition is a poset, denoted $\Gamma$ and viewed as a category. On the other hand, the second category, denoted by $k\Gamma_0$, is not the category of a finite poset. It is, in some sense, the category of a poset with some extra zero-relations. More formally, the category of $k$-linear representations of $k\Gamma_0$ is equivalent to the category of modules over an algebra which is a quotient of the incidence algebra of $\Gamma$ by zero-relations. The main result of the article is that the categories of $k$-linear representations of $\Gamma$ and $k\Gamma_0$ are derived equivalent. This result is obtained as a special case of a slightly more general construction, that we illustrate with a few examples.

This general construction takes as starting point a pair of posets
$X,Y$ and a morphism from $Y$ to the distributive lattice of lower
ideals in $X$. In fact, we think that it might be seen as a very special case of some derived equivalences obtained by Asashiba, which has considered the so-called Grothendieck construction in \cite{asashiba}. Our results are much more elementary and concrete, with a shorter proof and provide an explicit and simple tilting complex. One can hope to apply them in many combinatorial contexts.

The category of intervals $k\Gamma_0$ seems to be a good intermediate object when one wants to produce derived equivalences between finite posets. As applications, we prove that the Auslander algebra of a linear order $A_n$ is derived equivalent to the incidence algebra of the poset of intervals of $A_n$. Together with results of the second author, this proves that the \emph{rectangle} poset $A_{2n+1}\times A_{n}$ is derived equivalent to the \emph{triangle} poset of intervals of $A_{2n}$. Finally, we investigate the relations between the derived category of the poset of $(a,b)$-\emph{rational Dyck} paths and the poset of \emph{lattice paths} in the $(a,b)$-rectangle. 
\begin{notations}
If $\mathcal{A}$ is an abelian category, we denote by $D^b(\mathcal{A})$ its bounded derived category. We denote by $\mathrm{proj}(\mathcal{A})$ the full subcategory of $\mathcal{A}$ consisting of the finitely generated projective objects. If $\mathcal{B}$ is an additive category, we denote by $K^{b}(\mathcal{B})$ the homotopy category of bounded complexes of $\mathcal{B}$. 
\newline If $n\in\mathbb{N}$, we denote by $\overrightarrow{A_n}$ the set $\{1,\cdots, n\}$. Unless specified otherwise, we see it with the total order $1<2 <\cdots < n$.
\end{notations}
\section{Two categories of generalized intervals of a poset}
\subsection{Categories of intervals of a finite poset}
Let $k$ be a commutative ring with unit. Let $(X,\leqslant)$ be a finite poset. For $a,b\in X$, we set
\[ [a,b]:= \{z\in X; a\leqslant z \hbox{ and } z\leqslant b \}.\]
As usual, the set $[a,b]$ is called an interval of $X$. We let $\Int(X)$ be the set of intervals of $X$. It has a natural partial order defined by
\[ [a,b] \leqslant [c,d] \hbox{ if and only if } a\leqslant c \hbox{ and } b\leqslant d.\]
\begin{de}
The set $\Int(X)$ with this particular partial order is called the poset of intervals of $X$. 
\end{de}
\begin{re}
There is another natural partial order on the set $\Int(X)$ which is given by the inclusion of the intervals. However, with this partial order, the resulting endo-functor of the category of finite posets behaves less nicely. For example, it does not commutes with the duality. One can see that this poset is not equivalent to the poset that we consider in this article, even at the level of derived categories. 
\end{re}
If $[a,b]$ is an interval of $X$, then we have an indecomposable functor $M_{a,b}$ in $\mathcal{F}_{X,k}$ defined on the objects by:
\[ M_{a,b}(x) = \left\{\begin{array}{c}\ \ k \hbox{\ \ \ \ if $a\leqslant x\leqslant b$,}\\ 0\ \ \ \hbox{  otherwise.} \end{array}\right.\]
If $\alpha : x\to y$ is a morphism in $k\mathcal{C}_X$, then $M_{a,b}(\alpha) = \alpha$ if $a\leqslant x\leqslant y \leqslant b$ and $M_{a,b}(\alpha)=0$ otherwise. 
\begin{de}
Let $X$ be a finite poset and $k$ be a commutative ring. We let $\mathrm{Int}^{0}_k(X)$ be the category where the objects are the intervals of $X$ and, \[\Hom_{\mathrm{Int}^{0}_k(X)}([c,d],[a,b]):=\Hom_{\mathcal{F}_{X,k}}\big(M_{a,b},M_{c,d}\big).\]
\end{de}
Note that we applied an `op'-functor on each set of morphisms. 
\begin{lemma}
Let $X$ be a finite poset, let $k$ be a commutative ring. Then
\[ \Hom_{\mathcal{F}_{X,k}}(M_{a,b},M_{c,d}) = \left\{\begin{array}{c}k \hbox{ if $c\leqslant a\leqslant d\leqslant b$,} \\0 \hbox{ otherwise. }\end{array}\right.\]
\end{lemma}
\begin{proof}
If there is a non-zero morphism $\phi$ between $M_{a,b}$ and $M_{c,d}$, then the intersection $[a,b]\cap [c,d]$ is non-empty. Let $x$ be an element of this intersection. Since $\phi$ is a natural transformation, the following diagram commutes:
\[ \xymatrix{
M_{a,b}(x) = k \ar[r]^{\phi_x} & M_{c,d}(x)=k \\
M_{a,b}(a) = k \ar[r]^{\phi_a}\ar[u] & M_{c,d}(a). \ar[u] 
}\]
This implies that $M_{c,d}(a)\neq 0$. So we have $c\leqslant a\leqslant d$. Similarly, the following diagram commutes:
\[ 
\xymatrix{
M_{a,b}(d)\ar[r]^{\phi_d} & M_{c,d}(d)=k \\
M_{a,b}(x)=k \ar[r]^{\phi_x}\ar[u]& M_{c,d}(x)=k.\ar[u]}
\]
So, $M_{a,b}(d)\neq 0$. This implies that $a\leqslant d \leqslant b$. Conversely, if $c\leqslant a\leqslant d\leqslant b$, then the morphism $\phi$ defined by $\phi_{z}=\operatorname{Id}_k$ for every $a\leqslant z\leqslant d$ is a natural transformation from $M_{a,b}$ to $M_{c,d}$. 
\end{proof}
In other terms, we have a combinatorial description of $\mathrm{Int}^{0}_k(X)$:
\begin{coro}
Let $X$ be a finite poset. Let $k$ be a commutative ring. Then $\mathrm{Int}^{0}_k(X)$ is the category where the objects are the intervals of $X$ and the morphisms are:
\[\Hom_{\mathrm{Int}^{0}_k(X)}([a,b],[c,d])= \left\{\begin{array}{c}k \hbox{ if $a\leqslant c\leqslant b\leqslant d$,} \\0 \hbox{ otherwise. }\end{array}\right.\]
The composition is given by scalar multiplication. 
\end{coro}
It is easy to see that the categories of $k$-linear representations of $k\mathrm{Int}(X)$ and $\mathrm{Int}^{0}_k(X)$ are not equivalent. This is already the case when $X = \overrightarrow{A_2}$. However, we will see that they share the same derived category. 
\subsection{Generalized intervals of a finite posets}
Let $(X,\leqslant)$ be a finite poset. For an element $x\in X$, we let $[.,x] = \{ x'\in X\ ; \ x'\leqslant x\}$. A subset $Z\subseteq X$ is \emph{closed} if $[.,x] \subseteq Z$ for every $x\in Z$. We denote by $\mathcal{J}(X)$ the poset of closed subsets of $X$ partially ordered by inclusion. Note that a closed subset of $X$ is also called an \emph{ideal} of $X$. 
\newline\indent Let $X$ and $Y$ be two finite posets. Let $F : Y \to \mathcal{J}(X)$ be an order-preserving map. In other words, for $y\in Y$ there is a closed subset $F(y)$ of $X$ such that $F(y)\subseteq F(y')$ whenever $y\leqslant y'$. Consider $\Gamma$ the poset defined by
\[ \Gamma = \bigsqcup_{y\in Y}\big(F(y)\times \{y\} \big) \subseteq X\times Y\]
with the partial order induced from that of the product $X\times Y$. In other terms, the elements of $\Gamma$ are pairs $(x,y)$ where $y\in Y$ and $x\in F(y)$, with
\[ (x,y)\leqslant (x',y') \Leftrightarrow x \leqslant x' \hbox{ and } y\leqslant y'. \]
The elements of $\Gamma$ are called \emph{generalized intervals} for the data $(X,Y,F)$.
\newline\indent Let us consider the $k$-linear category $k\Gamma_0$ where the objects are the pairs $(x,y)$ such that $y\in Y$ and $x\in F(y)$ and the morphisms are given by
\[\Hom_{k\Gamma_0}\big((x,y),(x',y')\big) =\left\{\begin{array}{c}k \hbox { if $x\leqslant x'$, $y\leqslant y'$ and $x'\in F(y)$}, \\0 \hbox{ otherwise.} \end{array}\right.  \] 
The composition is given by the scalar multiplication. 
\begin{de}
Let $X$ and $Y$ be two finite posets and $F : Y \to \mathcal{J}(X)$ be an order preserving map. Let $k$ be a commutative ring. Then, we denote by $\mathcal{F}_{\Gamma_0,k}$ the category of functors from $k\Gamma_0$ to $k\Mod$. 
\end{de}
\begin{re}
This setting is a generalization of the two previous constructions for the intervals of a given poset $X$. Indeed,
if $X=Y=Z$ and $F : X\to \mathcal{J}(X)$ is the map defined by $F(x)=[\cdot, x]$, then $\Gamma = \mathrm{Int}(X)$ and $k\Gamma_0 \cong \mathrm{Int}^0_k(X)$. 
\end{re}
\section{Main result and applications}
\subsection{Main Theorem}
\begin{theo}\label{main_thm}
Let $k$ be a commutative ring. Let $X$ and $Y$ be two finite posets and $F : Y \to \mathcal{J}(X)$ be an order preserving map. Then, there is a triangulated equivalence between $D^{b}\big(\mathcal{F}_{\Gamma,k}\big)$ and $D^{b}\big(\mathcal{F}_{\Gamma_0,k}\big)$. 
\end{theo}
\begin{re}
We postpone the proof until Sections $4$ and $5$. 
\end{re}
Let us give an equivalent formulation of $\Gamma$ and the category $k\Gamma_0$ which is easier to manipulate. Let $X$, $Y$ and $Z$ be three finite posets with order preserving maps $f : X \to Z$ and $g : Y \to Z$. Consider the poset
\[\Gamma = \{ (x,y) \in Z\ ;\ f(x)\leqslant_{Z} g(y)\} \]
with partial order induced from $X\times Y$. 
\newline\indent Let $k\Gamma_0$ the category where the objects are the elements of $\Gamma$ and the morphisms are given by
\[\Hom_{k\Gamma_0}\big((x,y),(x',y')\big) =\left\{\begin{array}{c}k \hbox { if $x\leqslant_X x'$, $y\leqslant_Y y'$ and $f(x')\leqslant_Z g(y)$}, \\0 \hbox{ otherwise.} \end{array}\right.  \] 
Let $X$ and $Y$ be two finite posets and $F : Y\to \mathcal{J}(X)$ be an order preserving map. We set $Z= \mathcal{J}(X)$, the map $f : X\to J(X)$ is defined by $f(x)= [\cdot,x]$ and $g=F$. Then, the condition on a pair $(x,y)\in X\times Y$ that $f(x)\leqslant_{Z} g(y)$ means that $[\cdot,x] \subseteq F(y)$. Since $F(y)$ is closed, this is equivalent to the condition $x\in F(y)$. 
\newline Conversely, let $X$, $Y$ and $Z$ be three finite posets. Let $f : X \to Z$ and $g : Y\to Z$ be two order preserving maps. Then, for $y\in Y$, we let $F(y) = \{ x\in X\ ;\ f(x)\leqslant_Z g(y) \}$. Since $f$ is order preserving, the set $F(y)$ is closed and since $g$ is order preserving, we have $F(y)\subseteq F(y')$ whenever $y\leqslant_Y y'$. The condition on a pair $(x,y)\in X\times Y$ that $f(x)\leqslant_Z g(y)$ is equivalent to the condition that $x\in F(y)$. 
\begin{ex}
As first application, we consider some simple cases. 
\begin{enumerate}
\item Let $X = \{1,2,3\}$ such that $1< 3$ and $2< 3$. Let $Y = \{a,b,c,d\}$ such that $a < b < c < d$. Let $Z = \{i,j,k\}$ such that $i< j < k$. The morphism $f$ is defined by $f(1)= i$, $f(2)=j$ and $f(3)=k$. The morphism $g$ is defined by $g(a)=i$, $g(b)=j$, $g(c)=g(d)=k$. Then the Hasse diagram of $\Gamma$ is
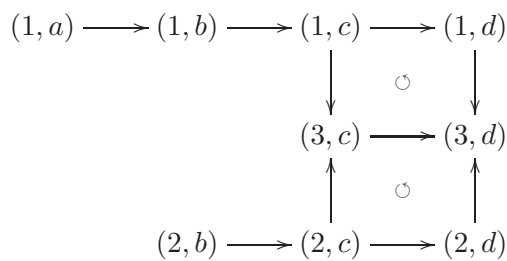
\[ 
\xymatrix{
(1,a) \ar[r] & (1,b) \ar[r] & (1,c) \ar[r]\ar[d]\commutatif & (1,d)\ar[d] \\
& & (3,c) \ar[r]\commutatif & (3,d) \\
& (2,b) \ar[r] &(2,c) \ar[r]\ar[u] &(2,d) \ar[u]
}
\]
For $k\Gamma_0$ we have the following presentation with generators and relations of the category
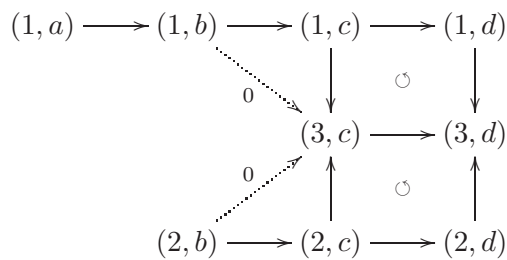
\[ 
\xymatrix{
(1,a) \ar[r] & (1,b) \ar[r]\ar@{..>}[rd]_{0} & (1,c) \ar[r]\ar[d]\commutatif & (1,d)\ar[d] \\
& & (3,c) \ar[r]\commutatif & (3,d) \\
& (2,b) \ar[r]\ar@{..>}[ru]^{0} &(2,c) \ar[r]\ar[u] &(2,d) \ar[u]
}
\]
where the dotted arrows are zero relations. 
\item It is particularly interesting to consider the more symmetric case where $Y=Z$ and $g= Id_{Y}$. Then 
\[ \Gamma = \{ (x,y)\in X\times Y\ ;\ f(x)\leqslant_{Y} y \}\]
and the category $k\Gamma_0$ has morphisms:
\[\Hom_{k\Gamma_0}\big((x,y),(x',y')\big) =\left\{\begin{array}{c}k \hbox { if $x\leqslant_X x'$, $y\leqslant_Y y'$ and $f(x')\leqslant_Y y$}, \\0 \hbox{ otherwise.} \end{array}\right.  \] 
Let $P = \{1,2,3\}$ where $1<3$ and $2<3$. Let $X$ be the poset of $2$-chains of $P$ and $Y$ be the poset of $3$-chains of $P$. We let $f : X\to Y$ to be the morphism that sends a chain $i\leqslant j$ to $i\leqslant i \leqslant j$. Then, the Hasse diagram of $\Gamma$ is
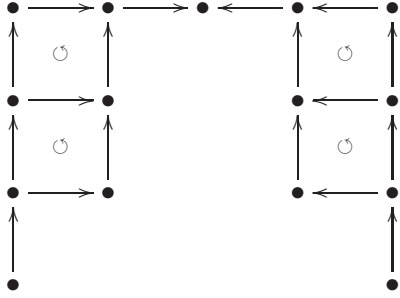
\[ 
\xymatrix{
\bullet\ar[r]\commutatif & \bullet \ar[r] & \bullet & \bullet\ar[l]\commutatif & \bullet \ar[l] \\
\bullet\ar[r]\commutatif\ar[u] & \bullet\ar[u] & & \bullet\ar[u] \commutatif& \bullet\ar[u]\ar[l]\\
\bullet\ar[r]\ar[u] & \bullet\ar[u] & & \bullet\ar[u] & \bullet\ar[u]\ar[l]\\
\bullet\ar[u] & & & & \bullet\ar[u]}
\]
For $k\Gamma_0$ we have a presentation by generators and relations
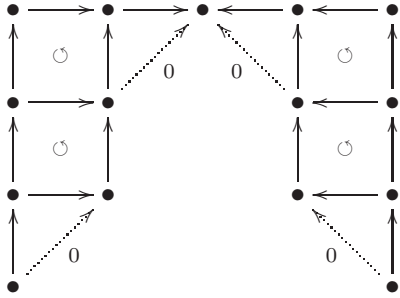
\[ 
\xymatrix{
\bullet\ar[r]\commutatif & \bullet \ar[r] & \bullet & \bullet\ar[l]\commutatif & \bullet \ar[l] \\
\bullet\ar[r]\ar[u]\commutatif & \bullet\ar[u]\ar@{..>}[ur]_{0} & & \bullet\ar[u]\ar@{..>}[ul]^{0} \commutatif& \bullet\ar[u]\ar[l]\\
\bullet\ar[r]\ar[u] & \bullet\ar[u] & & \bullet\ar[u] & \bullet\ar[u]\ar[l]\\
\bullet\ar[u]\ar@{..>}[ur]_{0} & & & & \bullet\ar[u]\ar@{..>}[ul]^{0}
}
\]
More generally, if $l \in \mathbb{N}$, one can defined simplicial morphisms between the poset of $l$-chains of $P$ and the poset of $l+1$-chains of $P$, by duplicating or forgetting the element at a fixe position of the chains. This will give similar diagrams. 
\end{enumerate}
\end{ex}
\subsection{Special case of intervals}
For the specific case of the intervals of a finite poset, we have
\begin{coro}
Let $X$ be a finite poset. Let $k$ be a commutative ring. Then, the category $\Fintk$ is derived equivalent to the category $\mathcal{F}_{\mathrm{Int}_k^{0}(X),k}$. 
\end{coro}
It was shown by the second author that the poset $A_{2n+1} \times A_n$ is derived equivalent to the stable Auslander algebra of the quiver $\overrightarrow{A_{2n+1}}$. The poset $A_{2n+1} \times A_n$ can be viewed as a \emph{rectangle}, and the stable Auslander algebra of the quiver $\overrightarrow{A_{2n+1}}$ with linear order can be seen as a \emph{triangle}. However, there are some zero-relations in the Auslander Reiten quiver that come from the almost split sequences where the left and the right terms are two simple modules. Using Theorem \ref{main_thm}, we can remove these zero-relations. 
\begin{coro}\label{triangle}
Let $n\in \mathbb{N}$ and $k$ be an algebraically closed field. Then, the poset $A_{2n+1} \times A_n$ is derived equivalent to the poset $\mathrm{Int}(\overrightarrow{A_{2n}})$.
\end{coro}
\begin{proof}
By Corollary $1.12$ of \cite{ladkani_rectangles}, the poset $A_{2n+1}\times A_n$ is derived equivalent to the stable Auslander algebra of $k\overrightarrow{A_{2n+1}}$. Because we consider a linear order on $A_{2n+1}$, it is easy to see that the stable Auslander algebra of $k\overrightarrow{A_{2n+1}}$ is isomorphic to the usual Auslander algebra of $k\overrightarrow{A_{2n}}$. Now, for $m\in \mathbb{N}^{*}$, we consider the Auslander algebra of $A_m$ with ordering $m < m-1 < \cdots < 1$. Let $Q$ be the Auslander Reiten quiver of $kA_m$ viewed as a category. Let $I$ be the category $\mathrm{Int}_{k}^{0}(\overrightarrow{A_m})$. There is a functor from $I$ to $Q$ which can be described as follows. The interval $[i,j]$ is sent to the indecomposable $kA_m$ module with support $[i,j]$, denoted $M_{[i,j]}$. If $[i,j] \leqslant [k,l]$, then the corresponding basis element is sent to the irreducible morphism between the indecomposable modules $M_{[i,j]}$ and $M_{[k,l]}$. 
\newline\indent If $i\neq j$, then the equality of the morphisms $[i,j] \to [i+1,j] \to [i+1,j+1]$ and $[i,j] \to [i,j+1] \to [i+1,j+1]$ corresponds via $\phi$ to the mesh relation 
{\small \[ 
\xymatrix{ & M_{[i+1,j]}\ar[rd] &  \\ M_{[i,j]}\ar[ru]\ar[rd] & & M_{[i+1,j+1]}=\tau^{-1}(M_{[i,j]}) \\ & M_{[i,j+1]}\ar[ru]
}
\]}
and the zero-relation $[i,i] \to [i,i+1] \to [i+1,i+1]$ in $\mathrm{Int}_k^{0}(A_k)$ corresponds to the mesh relation
{\small \[ \xymatrix{ 0 \ar[r] & M_{[i,i]} \ar[r] & M_{[i,i+1]} \ar[r] & M_{[i+1,i+1]} = \tau^{-1}(M_{[i,i]})\ar[r] & 0.}\]}
It is now easy to see that this functor is an equivalence of categories. In particular, the category of modules over the Auslander algebra of $A_m$ is equivalent to the category of $k$-linear functors from $\mathrm{Int}_k^{0}(\overrightarrow{A_m})$ to $k\Mod$. 
\newline\indent In conclusion, the poset $A_{2n+1} \times A_{n}$ is derived equivalent to $\mathrm{Int}_{k}^{0}(\overrightarrow{A_{2n}})$. The result follows from Theorem \ref{main_thm}. 
\end{proof}
\begin{re}
It is well-known that any two different orientations of a Dynkin diagram of type $A$ are derived equivalent. The stable Auslander algebra of two different orientations are also derived equivalent (see Section $1.5$ of \cite{ladkani_rectangles}). This implies that the poset of intervals $\mathrm{Int}(A_{2n})$ of a linear orientation of $A_{2n}$ is derived equivalent to the stable Auslander algebra of $A_{2n+1}$ for any orientation. However, It is wrong that two different orientations of $A_{2n}$ lead to derived equivalent posets of intervals. 
\end{re}
\subsection{Application to the poset of rational Dyck paths}
Let $a$ and $b$ be two co-prime integers. A rational $(a,b)$-Dyck path is a lattice path in an $(a\times b)$-rectangle that stays above and never crosses the diagonal. We denote by $\mathrm{Dyck}_{a,b}$ the set of rational Dyck paths. It is well known that there are $\frac{1}{a+b} {{a+b}\choose{b}}$ elements in $\mathrm{Dyck}_{a,b}$ (see \cite{bizley} for more details). The usual proof of this formula is to consider the set of all lattice paths in the rectangle $a\times b$, denoted by $\mathcal{L}_{a,b}$. This is a set with ${a+b}\choose{b}$ elements. The cyclic group $\mathbb{Z}/{(a+b)\mathbb{Z}}$ acts on this set by the so-called cycling relation of a path. The orbits contains $a+b$ elements and exactly one rational Dyck path. 
\newline\indent The sets $\mathcal{L}_{a,b}$ and $\mathrm{Dyck}_{a,b}$ can be naturally viewed as posets. A lattice path $l_1$ is smaller than another $l_2$ if $l_1$ lies below $l_2$. The formula for the cardinality of $\mathrm{Dyck}_{a,b}$ suggests a relation between the poset $A_{a+b}\times \mathrm{Dyck}_{a,b}$ and the poset $\mathcal{L}_{a,b}$. It is easy to see that these two posets are not isomorphic. Still, we think that they may share the same derived category. 
\begin{conj}\label{conj_dyck}
Let $a,b$ be two co-prime integers. Let $\mathcal{L}_{a,b}$ be the poset of lattice paths in the rectangle $a\times b$ and $\mathrm{Dyck}_{a,b}$ be the poset of $(a,b)$-rational Dyck paths. Then, the poset $A_{a+b} \times \mathrm{Dyck_{a,b}}$ is derived equivalent to the poset $\mathcal{L}_{a,b}$. 
\end{conj}
In the particular case where $a=2$, the tools developed here together with results of the second author can be used in order to check this conjecture. 
\begin{prop}
Let $k$ be an algebraically closed field. Let $b$ be an \emph{odd} integer. Then, there is a derived equivalence over the field $k$ between the posets $A_{b+2} \times \mathrm{Dyck}_{2,b}$ and $\mathcal{L}_{2,b}$. 
\end{prop}
\begin{proof}
If $\lambda$ is a lattice path in the rectangle $2\times b$, we denote by $I(\lambda)$ the pair $(j,i)$ where $i$ is the abscissa of the first vertical move of the path and $j$ is the abscissa of the second vertical move. The path is characterized by the pair $I(\lambda)$. There are $\frac{b+1}{2}$ different $(2,b)$-Dyck paths that corresponds to the pairs $(0,0), (1,0),\cdots ,(\frac{b-1}{2},0)$. Moreover, the partial order of the paths is given by $(\frac{b-1}{2},0) < \cdots < (1,0) < (0,0)$. In other words, $\mathrm{Dyck}_{2,b} \cong A_{\frac{b+1}{2}}$.
\newline\indent If $(j,i)$ is the pair $I(\lambda)$ of a lattice path, it is clear that $i\leqslant j$. In particular $I(\lambda)$ can be seen as an interval of $A_{b+1}$ ordered by decreasing order. If $\lambda_1$ and $\lambda_2$ are two paths, it is easy to see that $\lambda_1 \leqslant \lambda_2$ if and only if $I(\lambda_1)\leqslant I(\lambda_2)$ in the poset of intervals. This shows that $\mathcal{L}_{2,b}$ is isomorphic to the poset of intervals of $A_{b+1}$. The result follows from Corollary \ref{triangle}. 
\end{proof}
\section{Representations of a finite poset}
Let $k$ be commutative ring. Let $Y$ be a finite poset. The category $\FYk$ is abelian. The abelian structure is point-wise. More precisely, it is defined on the evaluations of the functors. For $y\in Y$, there is an obvious functor, denoted by $\ev_y$ from $\FYk$ to $k\Mod$ that sends a functor $F$ to its value $F(y)$. This functor is clearly exact. 
\newline\indent For $y \in Y$, we let $\mathrm{P}_y:= \Hom_{k\mathcal{C}_Y}(y,-)$. By Yoneda's Lemma, we have
\[ \Hom_{\FYk}\Big(\Hom_{k\mathcal{C}_Y}\big(y,-\big),-\Big) \cong \mathrm{ev}_y.\] In particular, the functor $\mathrm{P}_y$ is projective. Similarly, the functor $\mathrm{I}_y:=\Hom_{k\C_Y}(-,y)^{*}$ is an injective functor. More precisely, the evaluations of these functors are
\[ \mathrm{P}_y(z) = \left\{\begin{array}{c}k \hbox{ \ \ if $y \leqslant z$, }  \\0\hbox{ \ otherwise,}\end{array}\right. \]
\[ \mathrm{I}_y(z) = \left\{\begin{array}{c}k \hbox{ \ \ if $z \leqslant y$, }  \\0\hbox{ \ otherwise.}\end{array}\right. \]
\begin{lemma}\label{hom_proj}
Let $Y$ be a finite poset and $k$ be a commutative ring. Let $x$ and $y\in Y$. Then,
\[ \Hom_{\FYk}\big(\mathrm{P}_x, \mathrm{P}_y\big) = \Hom_{\FYk}\big(\mathrm{I}_x, \mathrm{I}_y\big)  = \left\{\begin{array}{c}k \hbox{ if $y \leqslant x$, } \\0 \hbox{ otherwise.}\end{array}\right.\]

\end{lemma}
\begin{proof}
These are straightforward applications of Yoneda's Lemma. 
\end{proof}
Let $X$ and $Y$ be two finite posets and $F : Y\to \mathcal{J}(X)$ be an order preserving map. For $y\in Y$, we let $i_{y} : F(Y) \to \Gamma$ be the map that sends $x\in F(y)$ to $(x,y)\in \Gamma$. This is an order preserving map, so the pre-composition by $i_{y}$ gives a functor $i_{y}^{-1} : \FGk \to \Fyk$. More explicitly, if $\phi \in \FGk$, then $i_{y}^{-1}(\phi)$ is the functor that sends $x\in F(y)$ to the $k$-module $\phi(x,y)$. The functor $i_{y}^{-1}$ is clearly exact and by usual arguments it has a left and a right adjoint which can be described as particular coend and end (for more details see Theorem $1$, Section $4$ of Chapter $X$ of \cite{cftwm}). One can explicitly compute this end in order to find the right adjoint. Alternatively, and for the convenience of the reader, we give the formula and check that this gives indeed a right adjoint of $i_{y}^{-1}$.
\newline\indent Let $\phi \in \Fyk$. Then, for $(a,b)\in \Gamma$ we set:
\[ (i_{y})_{\star}\phi(a,b):=\left\{\begin{array}{c}\Hom_{k}\Big(\Hom_{k\Gamma}\big((a,b),(a,y)\big), \phi(a)\Big) \hbox{ if $b \leqslant y$, } \\0 \hbox{ otherwise. }\end{array}\right.  \]
Since $\Hom_{k\Gamma}\big((a,b),(a,y)\big)$ is isomorphic to $k$ when $b\leqslant y$, this formula can be simplified as
\[ (i_{y})_{\star}\phi(a,b)\cong \left\{\begin{array}{c}\phi(a) \hbox{ if $b \leqslant y$, } \\0 \hbox{ otherwise. }\end{array}\right.  \]
However, we feel that it is more natural to describe this functor in this way.
\newline\indent Let $0\neq f : (a,b) \to (c,d)$ be a morphism in $k\Gamma$ such that $b\leqslant d \leqslant y$. Let $0\neq g \in \Hom_{k\Gamma}\big((c,d),(c,y) \big)$. Then, there exist $h\in \Hom_{k\Gamma}\big((a,b),(a,y)\big)$ and $\alpha \in \Hom_{kF(y)}(a,c)$ such that the following diagram commutes
\[
\xymatrix{
(a,b)\ar[r]^{h}\ar[d]^{f} & (a,y)\ar[d]^{i_{y}(\alpha)} \\
(c,d) \ar[r]^{g} & (c,y)
}
\]
Note that $h$ and $\alpha$ are not unique. However, the different choices are of the form $\lambda \times h$ and $\lambda^{-1}\times \alpha$ for $\lambda \in k^{\times}$.
\newline\indent Then $(i_{y})_{\star}(\phi)(f)$ is the application that sends $\rho\in \Hom_{k}\Big(\Hom_{k\Gamma}\big((a,b),(a,y)\big), \phi(a)\Big)$ to the $k$-linear morphism that sends $g\in \Hom_{k\Gamma}\big((c,d),(c,y)\big)$ to $\phi(\alpha) \circ \rho(h)\in \phi(c)$. Since $\phi$ and $\rho$ are $k$-linear morphisms, we see that the value of $(i_{y})_{\star}(\phi)(f)$ does not depend on the choice of $\alpha$ and $h$. 
\newline\indent Let $\eta : \phi \Rightarrow \psi$ be a morphism between two functors of $\Fyk$. Then $(i_{y})_{\star}(\eta)$ is the natural transformation defined by $(i_{y})_{\star}(\eta)_{(a,b)}\big(\rho) = \eta_{a} \circ \rho$ for $(a,b)\in \Gamma$ such that $b\leqslant y$ and $\rho \in \phi(a,b)$.
\begin{lemma}
Let $y\in Y$. Then, the functor $(i_{y})^{-1} : \FGk\to \Fyk$ is a left adjoint to the functor $(i_y)_{\star} : \Fyk \to \FGk$. 
\end{lemma}
\begin{proof}
We give the unit and the co-unit of the adjunction. 
\newline\indent Let $F\in \Fyk$ and $x\in F(y)$. The unit at $F$ and $x$, is the $k$-linear morphism 
\[\epsilon_F(x) : \Hom_{k}\Big(\Hom_{k\Gamma}\big((x,y),(x,y)\big),F(x)\Big)  \to F(x),\]
that sends $\alpha \in \Hom_{k}\Big(\Hom_{k\Gamma}\big((x,y),(x,y)\big),F(x)\Big)$ to $\alpha(Id_{(x,y)}) \in F(x)$. 
\newline\indent Let $G\in \FGk$ and $(a,b)\in\Gamma$ such that $b\leqslant y$. The co-unit of the adjunction at $G$ and $(a,b)$ is the $k$-linear morphism
\[\eta_{G}(a,b) : G(a,b) \to \Hom_{k}\Big(\Hom_{k\Gamma}\big((a,b),(a,y)\big),G(a,y)\Big) \]
that sends $\gamma \in G(a,b)$ to the $k$-linear morphism that sends $\alpha \in \Hom_{k\Gamma}\big((a,b),(a,y)\big)$ to $G(\alpha)(\gamma)$. It is now straightforward to check that these two morphisms are the unit and the co-unit of the adjunction. 
\end{proof}
Here, we summarise the main properties of the functors $i_{y}^{-1}$ and $(i_y)_{\star}$.
\begin{lemma}\label{adj}
Let $y\in Y$.
\begin{enumerate}
\item The functor $(i_y)_{\star}$ sends the injective $I_{x} \in \Fyk$ to the injective $I_{(x,y)} \in \FGk$.
\item The two functors $i_y^{-1}$ and $(i_y)_{\star}$ are exact. 
\end{enumerate}
\end{lemma}
\begin{proof}
Since $(i_y)_{\star}$ is a right-adjoint to an exact functor, it sends $I_{x}\in \Fyk$ to an injective object of $\FGk$. Moreover, one can explicitly compute $(i_y)_{\star}(I_{x})$. Let $(a,b)\in \Gamma$, then we have
\[(i_y)_{\star}(I_x)(a,b) = \left\{\begin{array}{c} I_{x}(a) \hbox{ if $b\leqslant y$} \\0 \hbox{ otherwise } \end{array}\right. = \left\{\begin{array}{c} k \hbox{ if $a\leqslant x$ and $b\leqslant y$} \\0 \hbox{ otherwise. } \end{array}\right. \]
It is clear that $i_y^{-1}$ is an exact functor. For the functor $(i_{y})_{*}$ the exactness follows easily from the description of this adjoint. Let 
\[\xymatrix{
0 \ar[r] & F_1 \ar[r]^{\alpha_1} & F_2 \ar[r]^{\alpha_2} & F_3 \ar[r] &0,
} \]
be an exact sequence of functors of $\Fyk$. Let $(a,b)\in \Gamma$. If $b \nleqslant y$, then for $i=1,2,3$ we have $(i_{y})_{\star}(F_i)(a,b) = 0$  and $(i_{y})_{\star}(\alpha_i)_{(a,b)}=0$ for $i=1,2$ so the sequence is exact. If $b\leqslant y$, then $(i_{y})_{\star}(\alpha)_{(a,b)} = \Hom_{k}\Big(\Hom_{k\Gamma}\big((a,b),(a,y)\big),(\alpha_i)_{a}\Big)$ for $i=1,2$. Since $\Hom_{k\Gamma}\big((a,b),(a,y)\big)\cong k$, the result follows. 
\end{proof}
Since the functors $(i_y)_{\star}$ and $i_y^{-1}$ are both exact, they can be extended as a pair of adjoint triangulated functors between the derived categories $D^{b}\big(\Fyk\big)$ and $D^b\big(\FGk\big)$.
\newline\indent Let us remark that, in general, the functor $(i_{y})_{\star}$ does not send a projective functor to a projective functor. However, it behaves relatively nicely for these projective functors.
\begin{lemma}
Let $y$ and $y'\in Y$. Let $x\in F(y)$ and $P_x$ be the corresponding projective functor when it is followed by a restriction. Then,

\[ \big(i_{y'}^{-1} \circ (i_{y})_{\star}\big)(P_x)\cong \left\{\begin{array}{c}P_x \in \mathcal{F}_{F(y'),k} \hbox{ if $ y'\leqslant y$ and $x\in F(y')$, } \\0 \hbox{ otherwise. }\end{array}\right.  \]
\end{lemma} 
\begin{proof}
Let $x'\in F(y')$. Then, we have
\begin{align*}
 \big(i_{y'}^{-1} \circ (i_{y})_{\star}\big)(P_x)(x') &=(i_{y})_{\star}(P_x)(x',y') \\
 &= \left\{\begin{array}{c}P_x(x') \hbox{ if $ y'\leqslant y$, } \\0 \hbox{ otherwise. }\end{array}\right. \\
 &= \left\{\begin{array}{c}k \hbox{ if $ y'\leqslant y$, and $x \leqslant x'$ } \\0 \hbox{ otherwise. }\end{array}\right.
\end{align*}
Since $F(y')$ is closed, the condition $x\leqslant x'$ implies that $x\in F(y')$. The result follows. 
\end{proof}
\section{Proofs of Theorem \ref{main_thm}}
First let us recall the famous Morita theorem for derived categories of Rickard.
\begin{theo}
Let $A$ and $B$ be two rings. Then, the following are equivalent
\begin{enumerate}
\item $D^b(A\Mod) \cong D^b(B\Mod)$.
\item $B$ is isomorphic to $\End_{D^b(A)}(T)^{op}$ where $T$ is an object of $K^{b}(\mathrm{proj}(A))$ satisfying
\begin{itemize}
\item $\Hom_{D^{b}(A)}(T,T[i]) = 0 $ for $i\neq 0$,
\item $\mathrm{add}(T)$, the category of direct summands of finite direct sums of copies of $T$, generates $K^b(\mathrm{proj}(A))$ as triangulated category.
\end{itemize}
\end{enumerate}
\end{theo}
\begin{proof}
See Theorem $6.4$ of \cite{rickard_morita} or Theorem $6.5.1$ of \cite{zimmermann_rep} for a proof following Keller's approach. Note that the proof of Keller is stronger. It shows that it is possible to realise the derived equivalence as the tensor product with a bounded complex of bimodules. However, it holds for two $k$-algebras over a commutative ring that are projective over $k$.   
\end{proof}
The complex $T$ is called a tilting complex for the ring $A$. In the present paper, we work with categories of functors. However, it is easy to see that all our functors categories are equivalent to categories of modules over an algebra. Moreover, these algebras are free over the ring $k$ and of finite rank. In particular, in order to build derived equivalences, we can use Rickard's Morita theorem. Since the algebras are free over $k$, the stronger form of this theorem due to Keller also holds in our context. 
\newline\indent Let $y\in Y$ and $x\in Y$. We denote by $P_{x}$ the projective functor of $\Fyk$ that corresponds to the element $x$.
\begin{prop}\label{tilting}
Let $k$ be a commutative ring. Let $X$ and $Y$ be two finite posets and let $F : Y \to \mathcal{J}(X)$ be an order preserving map. Then, the complex 
\[\mathrm{T} := \bigoplus_{y\in Y} \bigoplus_{x\in F(y)} (i_{y})_{\star} (P_{x}) \]
is a tilting complex for $D^{b}\big(\mathcal{F}_{\Gamma,k}\big)$. 
\end{prop}
\begin{proof}
Strictly speaking, the complex $\mathrm{T}$ is not a tilting complex since its terms are not projective. However, since the category $\mathcal{F}_{\Gamma,k}$ has finite global dimension, we can find a bounded complex of projective objects which is quasi-isomorphic to $\mathrm{T}$. 
\begin{enumerate}
\item Let $y \in Y$. Let $x\in F(y)$ and $I_{x}$ be a finitely injective functor of $\Fyk$ corresponding to $x$. It has a finite projective resolution, so there is a bounded complex $P_{\bullet} $ of elements of $add(\bigoplus_{x \in F(y)} P_x)$ and a quasi-isomorphism $\phi : P_{\bullet} \to I$. Since the functor $(i_{y})_{\star}$ is exact, we have a quasi-isomorphism between $(i_{y})_{\star}(P_{\bullet})$ and $(i_y)_{\star}(I)$. By Lemma \ref{adj}, we have $(i_y)_{\star}(I_{x}) \cong I_{(x,y)}$.  So for every $(x,y)\in \Gamma$, the injective functor $I_{(x,y)}$ belongs to the smallest triangulated subcategory of $D^{b}(\FGk)$ that contains $\mathrm{add(T)}$. Since every finitely projective functor has a finite injective co-resolution, we conclude that this category is equivalent to $K^{b}\big(\mathrm{proj}(\FGk)\big)$. 
\item  Let $i\in \mathbb{Z}$. Then, using the fact that $(i_{y'})_{\star}$ is a triangulated functor and the adjunction, we have
\begin{align*}
\Hom_{D^b(\FGk)}\big(\mathrm{T},\mathrm{T}[i]\big) &= \bigoplus_{y,y'\in Y} \Hom_{D^b(\FGk)}\Big((i_{y})_{\star}\big(\bigoplus_{x\in F(y)}P_{x}\big),(i_{y'})_{\star}\big(\bigoplus_{x'\in F(y')}P_{x'}\big)[i]\Big)\\
&\cong \bigoplus_{y,y'\in Y} \Hom_{D^b(\mathcal{F}_{F(y'),k})}\Big(i_{y'}^{-1}\circ (i_{y})_{\star}\big( \bigoplus_{x\in F(y)}P_{x}\big),\bigoplus_{x'\in F(y')}P_{x'}[i]\Big) \\
& \cong \bigoplus_{y'\leqslant y} \Hom_{D^b(\mathcal{F}_{F(y'),k})}\Big(\bigoplus_{x\in F(y')}P_{x},\bigoplus_{x'\in F(y')}P_{x'}[i]\Big)
\end{align*}
Since $P_x$ and $P_x'$ are projective objects in $\mathcal{F}_{F(y'),k}$, there are no non-trivial extensions between them. This implies that $\Hom_{D^b(\FGk)}\big(\mathrm{T},\mathrm{T}[i]\big) = 0$ if $i\neq 0$. 
\end{enumerate}
\end{proof} 
\begin{proof}[Proof of Theorem \ref{main_thm}]
By Proposition \ref{tilting}, there is an equivalence between $D^b(\FGk)$ and $D^b(\End(T)^{op})$, where $T$ is the tilting complex of the Proposition. By usual Morita theory, the category $\mathcal{F}_{\Gamma_0,k}$ is equivalent to $\End_{\mathcal{F}_{\Gamma_{0},k}}\big(\bigoplus_{(x,y)\in k\Gamma_0} P_{(x,y)}\big)^{op}$ where $P_{(x,y)}$ is the representable functor $\Hom_{k\Gamma_0}\big((x,y),-\big)$. Using the Yoneda Lemma, we have
\[ \Hom_{\mathcal{F}_{\Gamma_{0},k}}\big(P_{(x,y)},P_{(x',y')}\big) = \left\{\begin{array}{c} k \hbox{ if $x'\leqslant x$, $y'\leqslant y$ and $x\in F(y')$, } \\0 \hbox{ otherwise.}\end{array}\right. \]
Since $(i_{y})_{\star}(P_x)$ and $(i_{y'})_{\star}(P_{x'})$ for $x\in F(y)$ and $x\in F(y')$ and $y,y'\in Y$, are two functors, we can do the computation of $\End(\mathrm{T})$ in the categories of functors instead of the derived category. Then, we have
\begin{align*}
\Hom_{\FGk}\big( (i_{y})_{\star}(P_x), (i_{y'})_{\star}(P_{x'})\big) & \cong \Hom_{\mathcal{F}_{F(y'),k}}\big(i_{y'}^{-1} \circ (i_{y})_{\star}(P_x), P_{x'}\big)\\
& \cong  \left\{\begin{array}{c} \Hom_{\mathcal{F}_{F(y'),k}}(P_{x},P_{x'}) \hbox{ if $y'\leqslant y$ and $x\in F(y')$} \\0 \hbox{ otherwise,}\end{array}\right.\\
& \cong \left\{\begin{array}{c} k \hbox{ if $x'\leqslant x$, $y'\leqslant y$ and $x\in F(y')$, } \\0 \hbox{ otherwise.}\end{array}\right.
\end{align*}
This implies that $End(T) \cong \End_{\mathcal{F}_{\Gamma_{0},k}}\big(\bigoplus_{(x,y)\in k\Gamma_0} P_{(x,y)}\big)$. Taking the `op' functor, we have the derived equivalence between $\FGk$ and $\mathcal{F}_{\Gamma_0,k}$. 
\end{proof}

\bibliographystyle{alpha}

\begin{thebibliography}{}

\end{thebibliography}


\begin{thebibliography}{Cha12}

\bibitem[Asa13]{asashiba}
H. Asashiba.
\newblock Gluing derived equivalences together.
\newblock {\em Advances in Mathematics}, 235:134--160, 2013.

\bibitem[Biz54]{bizley}
M.~T.~L. Bizley.
\newblock Derivation of a new formula for the number of minimal lattice paths
  from {$(0,0)$} to {$(km,kn)$} having just {$t$} contacts with the line
  {$my=nx$} and having no points above this line; and a proof of {G}rossman's
  formula for the number of paths which may touch but do not rise above this
  line.
\newblock {\em J. Inst. Actuar.}, 80:55--62, 1954.

\bibitem[Cha12]{chapoton_derived_tamari}
F. Chapoton.
\newblock On the categories of modules over the {T}amari posets.
\newblock In {\em Associahedra, {T}amari lattices and related structures},
  volume 299 of {\em Prog. Math. Phys.}, pages 269--280. Birkh\"auser/Springer,
  Basel, 2012.

\bibitem[Lad07]{ladkani_universal}
S. Ladkani.
\newblock Universal derived equivalences of posets.
\newblock arXiv:0705.0946, 2007.

\bibitem[Lad08]{ladkani_derived_poset}
S. Ladkani.
\newblock On derived equivalences of categories of sheaves over finite posets.
\newblock {\em J. Pure Appl. Algebra}, 212(2):435--451, 2008.

\bibitem[Lad13]{ladkani_rectangles}
S. Ladkani.
\newblock On derived equivalences of lines, rectangles and triangles.
\newblock {\em J. Lond. Math. Soc. (2)}, 87(1):157--176, 2013.

\bibitem[Lan98]{cftwm}
S.M. Lane.
\newblock {\em Categories for the Working Mathematician}.
\newblock Graduate Texts in Mathematics. Springer New York, 1998.

\bibitem[Ric89]{rickard_morita}
J. Rickard.
\newblock Morita theory for derived categories.
\newblock {\em Journal of the London Mathematical Society}, 2(3):436--456,
  1989.

\bibitem[Sim92]{simson}
D. Simson.
\newblock {\em Linear representations of partially ordered sets and vector
  space categories}, volume~4 of {\em Algebra, Logic and Applications}.
\newblock Gordon and Breach Science Publishers, Montreux, 1992.

\bibitem[Zim14]{zimmermann_rep}
A.~Zimmermann.
\newblock {\em Representation Theory: A Homological Algebra Point of View}.
\newblock Algebra and Applications. Springer International Publishing, 2014.

\end{thebibliography}

{Fr\'ed\'eric Chapoton} \\
{Institut de Recherche Math\'ematique Avanc\'ee, CNRS UMR 7501, Universit\'e de Strasbourg, F-67084 Strasbourg Cedex, France} \\
{chapoton@unistra.fr}\\
{Sefi Ladkani}\\
{Department of Mathematics, University of Haifa, Mount Carmel,31905 Haifa, Israel}\\
{ladkani.math@gmail.com}\\
{Baptiste Rognerud} \\
{Fakultät für Mathematik Universität Bielefeld D-33501 Bielefeld, Germany} \\
{brognerud@math.uni-bielefeld.de}
\end{document}